\documentclass{amsart}
\usepackage{amsmath,amsthm,amsfonts,amssymb,latexsym,mathrsfs,graphicx}

\usepackage{hyperref}

\usepackage{enumerate}
\usepackage[shortlabels]{enumitem}
\usepackage{color}
\usepackage {tikz}
\usetikzlibrary {positioning}
\definecolor {processred}{cmyk}{0,0.96,0,0}

\newcommand{\cvd}{\hspace*{\fill}
	{\rm \hbox{\vrule height 0.2 cm width 0.2cm}}}
\renewcommand{\qed}{\cvd}
\newtheorem{theorem}{Theorem}[section]

\newtheorem{lemma}[theorem]{Lemma}
\newtheorem{remark}[theorem]{Remark}

\theoremstyle{definition}
\newtheorem{definition}[theorem]{Definition}

\newtheorem{example}[theorem]{Example}

\newtheorem*{PMa}{The proof of Theorem A}
\newtheorem*{PMb}{The proof of Theorem B}

\tikzstyle{vertex}=[circle, draw, inner sep=0pt, minimum size=6pt] 
\newcommand{\vertex}{\node[vertex]}

\newcommand{\Cay}{\mathrm{Cay}}

\makeatletter
\newcommand*{\rom}[1]{\expandafter\@slowromancap\romannumeral #1@}
\makeatother
\begin{document}
	
\title{ On the subgroup Regular sets in  cayley graphs}

\author[Y. Khaefi]{Yasamin Khaefi$^1$}
\author[Z. Akhlaghi]{Zeinab Akhlaghi$^{1,2}$}
\author[B. Khosravi]{Behrooz Khosravi$^1$}

\address{$^{1}$ Faculty of Mathematics and Computer Science, Amirkabir University of Technology (Tehran Polytechnic), 15914 Tehran, Iran.}
\address{$^{2}$ School of Mathematics,
	Institute for Research in Fundamental Science (IPM)
	P.O. Box:19395-5746, Tehran, Iran. }
\email{\newline \text{(Y. Khaefi) }khmath@aut.ac.ir \newline \text{(Z. Akhlaghi) }z\_akhlaghi@aut.ac.ir  \newline \text{(B. Khosravi) }khosravibbb@yahoo.com}

\thanks{
	The second author  is supported by a grant from IPM (No. 1402200112).}

\begin{abstract}
A subset $C$ of the vertex set of a graph $\Gamma$ is said to be $(a,b)$-regular if  $C$ induces an $a$-regular subgraph and every vertex outside $C$ is adjacent to exactly $b$ vertices in  $C$. In particular, if  $C$ is an $(a,b)$-regular set of some Cayley graph on a finite group $G$, then  $C$ is called an $(a,b)$-regular set of $G$ and a $(0,1)$-regular set is called a perfect code of $G$. In [Wang, Xia and Zhou, Regular sets in Cayley graphs, J. Algebr. Comb., 2022] it is proved that if $H$ is a normal subgroup of $G$, then $H$ is a perfect code of $G$ if and only if it is an $(a,b)$-regular set of $G$, for each $0\leq a\leq|H|-1$ and $0\leq b\leq|H|$ with $\gcd(2,|H|-1)\mid a$. In this paper, we generalize this result and show that a subgroup $H$ of $G$ is a perfect code of $G$ if and only if it is an $(a,b)$-regular set of $G$, for each $0\leq a\leq|H|-1$ and $0\leq b\leq|H|$ such that $\gcd(2,|H|-1)$ divides $a$. Also, in  [J. Zhang, Y. Zhu, A note on regular sets in Cayley graphs, Bull. Aust. Math. Soc., 2023] it is proved that if $H$ is a normal subgroup of $G$, then $H$   is an $(a,b)$-regular set of $G$, for each $0\leq a\leq|H|-1$ and  $0\leq b\leq|H|$ such that $\gcd(2,|H|-1)$ divides $a$ and $b$ is even. We extend this result and we prove that the normality condition is not needed. 

\end{abstract}


\maketitle

\section{Introduction}
In this paper, all groups are finite. If $\Gamma$ is a graph, we denote by ${\bf V}(\Gamma)$ and ${\bf E}(\Gamma)$ the set of its vertices and the set of its edges, respectively.   Let $\Gamma=({\bf V}(\Gamma), {\bf E}(\Gamma))$ be a simple graph. A subset $C$ of ${\bf V}(\Gamma)$ is called a perfect code of $\Gamma$, if every vertex of ${\bf V}(\Gamma)\setminus C$ is adjacent to exactly one vertex of $C$ and  there is no edges between vertices of $C$. Regular set is a generalization of perfect code of a graph \cite{R1} and is defined as follows: for non-negative integers $a, b$, a subset $C$ of ${\bf V}(\Gamma)$ is called an $\left(a,b \right) $- regular set in $\Gamma$,  if every vertex of ${\bf V}(\Gamma)\setminus C$ is adjacent to exactly $b$ vertices of $C$ and also every vertex of $C$ is adjacent  to exactly $a$ vertices of $C$. Clearly, a perfect code is a $(0,1)$-regular set. 

For a graph $\Gamma=({\bf V}(\Gamma), {\bf E}(\Gamma))$, a partition of ${\bf V}(\Gamma)$ with cells $\mathcal{V}=\{ V_1,\dots,V_k\}$ is called an equitable partition, when each cell induces a regular subgraph and edges between any two different cells construct a biregular bipartite graph \cite[Section 9.3]{godsil}. Equivalently, for $i\ne j$, any vertex of $V_i$, say $x$, is adjacent to $b_{ij}$ vertices of $V_j$, independent of the choice of $x$.  A  $k\times k$  matrix $M=(b_{ij})$ is called the quotient matrix of the partition $\mathcal{V}$. 

Let  $\Gamma$  be  a connected $ r$-regular graph. Then all row sums  of  the quotient matrix  $M$ is equal to
$r$, and so $ r$ is a simple eigenvalue of $M$ \cite[ Theorem 9.3.3]{godsil}. The equitable partition
$\mathcal{V}$ of $\Gamma$ is said to be $\mu$-equitable  if all eigenvalues of its quotient matrix $M$ other
than $r$ are equal to $\mu$. In \cite[ Corollary 2.3]{bcg} it is shown that a non-trivial coarsening
of a $\mu$-equitable partition is $\mu$-equitable.  So it is important to study equitable  partition with exactly two parts. 


Moreover,  an $(a, b)$-regular set in a $r$-regular graph
$\Gamma$ is exactly a completely regular code $C$ in $\Gamma$ (see, for example, \cite{n}]) such that the
corresponding distance partition has exactly two parts, namely $\{C, {\bf V} \setminus  C\}$. 
 An equitable partition with two parts is also called perfect 2-coloring \cite{2coloring}. The notion of perfect coloring is a common research subject in coding theory \cite{hamming2coloring1, hamming2coloring2}.

In this paper, we have studied regular sets in Cayley graphs. Let $G$ be a group and $S$ an inverse-closed subset of $G\setminus\left\lbrace 1\right\rbrace $. A Cayley graph $\Gamma=\Cay(G,S)$ is a graph whose vertex set is $G$ and two vertices $x,y\in G$ are adjacent if and only if $yx^{-1}\in S$. Perfect codes of Cayley graphs have been studied in many papers \cite{1, cyclic, 0, on, total} and specially because of their connection to Hamming codes, they are interesting objects.  When a subgroup of a  group $G$ is a perfect code of $\Cay(G,S)$, it is simply called a perfect code of $G$. In particular, authors investigated under which conditions a subgroup is a perfect code of the group \cite{on, paper}. Due to the link between perfect codes and regular sets, it is natural to study their mutual relation \cite{subregularset,regularset,note}. Authors of \cite{regularset} showed that a normal subgroup $H$ of a group $G$ is a perfect code of $G$ if and only if  $H$ is an $(a,b)$-regular set of $G$ for every pair of integers $a,b$ with $0\leq a\leq|H|-1$ and  $0\leq b\leq|H|$ such that $\gcd(2,|H|-1)\mid a$. In \cite{subregularset}, they proved the similar result for every subgroup perfect code of  a generalized dihedral group. In this paper, we improve the result and prove the following for any arbitrary subgroup perfect code $H$:

{\bf Theorem A}
 Let $H$ be a non-trivial  subgroup of $G$. 
 Then $H$ is a perfect code of $G$ if and only if it is an $(a,b)$-regular set of $G$ for every pair of integers $a,b$ with $0\leq a\leq|H|-1$ and $0\leq b\leq|H|$, such that $\gcd(2,|H|-1)\mid a$. 
 
 \smallskip
 
 In \cite{note}, it is shown that for  integer  $a$ and even integer $b$,  with $0\leq a\leq|H|-1$ and $0\leq b\leq|H|$, such that $\gcd(2,|H|-1)$ divides $a$, every normal subgroup $H$ is an  $(a,b)$-regular  set in $G$. As another result we cross out the normality hypothesis and we prove that:
 
 \smallskip
 
 {\bf Theorem B}
 Let $H$ be a non-trivial  subgroup of $G$. 
 Then, for  integer  $a$ and even integer $b$,  with $0\leq a\leq|H|-1$ and $0\leq b\leq|H|$, such that $\gcd(2,|H|-1)$ divides $a$,  $H$ is an  $(a,b)$-regular set  in $G$. 

 \smallskip

 Throughout the paper,  we use the following notations. If $L\subseteq  {\bf E}(\Gamma)$, then by ${\bf V}(L)$ we mean a subset of ${\bf V}(\Gamma)$ whose elements are the ends of the edges in $L$. If $V_0\subset {\bf V}(\Gamma)$, by $\Gamma[V_0]$ we mean the induced subgraph generated by $V_0$. 

\section{Main Results }

Throughout this section, let $G$ be a group,  $H$ a subgroup of $G$ and $x\in G\setminus H$. We denote
by $\Omega_x$ the set of right cosets of $H$ in $HxH\cup Hx^{-1}H$.  Firstly, we define two   graphs related to the cosets of $H$. 

\begin{definition} \label{def1} The simple graph (a graph without parallel edges and loops)  $\Delta_x$  is defined  as follows: 
The vertex set  of $\Delta_x$ consists of  elements of  $HxH\cup Hx^{-1}H$ and two distinct vertices $y$ and $z$ are adjacent if $yz=1$ and $Hy\not =Hz$.
Clearly, the degree of each vertex in  $\Delta_x$ is $0$ or $1$. 
\end{definition}

 \begin{definition}\label{def2}
The  non-simple graph $\Gamma_x$ is a graph with ${\bf V}(\Gamma_x)=\Omega_x$
and two distinct  cosets  $Hy$ and $Hz$  are adjacent with   $m$  parallel edges $\mathfrak{e}_{\{y,z\}}^i$,  for  $i=1,\dots,m$,  if $|z^{-1}H\cap Hy|=m$, i.e., the inverses of exactly $m$ elements in $Hy$ belong to $Hz$.  
\end{definition}

 Let $y\in HxH$ and $z\in Hx^{-1}H$ such that  $Hy\not =Hz$.  We will prove in Lemma \ref{key2}, that $z^{-1}H\cap Hy\ne \emptyset$. We define  $B_{\{y,z\}}=(z^{-1}H\cap Hy) \cup(z^{-1}H\cap Hy)^{-1} =(z^{-1}H\cap Hy) \cup (y^{-1}H\cap Hz)$.  By  Definition \ref{def1},  if   $z^{-1}H\cap Hy=\{t_1,\dots,t_m\}$ and $Hy\not = Hz$,   then  $B_{\{y,z\}}=\{t_1,\dots,t_m\}\cup \{t^{-1}_1,\dots, t^{-1}_m\}$ and  the induced subgraph $\Delta_x[B_{\{y,z\}}]$ is a perfect matching with $2m$ vertices and $m$  edges $\{t_i,t^{-1}_i\}$.  Let  $\mathfrak{E}_{\{y,z\}}=\{\mathfrak{e}_{\{y,z\}}^i\ | \ i=1,\dots , m\}$ be the set of  $m$ distinct edges between $Hy$ and $Hz$ in $\Gamma_x$. 
 We note that  $|\mathfrak{E}_{\{y,z\}}|=|{\bf E}(\Delta_x[B_{\{y,z\}}])|=m$.  Therefore,  $\phi_{\{y,z\}}:\mathfrak{E}_{\{y,z\}}\rightarrow {\bf E}(\Delta_x[B_{\{y,z\}}])$, which is defined by $\phi_{\{y,z\}}(\mathfrak{e}_{\{y,z\}}^i)=\{t_i,t^{-1}_i\}$,  is a bijection.

  Let $\phi$ be a function from ${\bf E}(\Gamma_x)$ to ${\bf E}(\Delta_x)$ such that   $\phi|_{\mathfrak{E}_{\{y,z\}}}=\phi_{\{y,z\}}$, for every $y\in HxH$ and $z\in Hx^{-1}H$,  where $Hy\not =Hz$.  Then, $\phi$ is a bijection.   In the rest of the paper we use the bijection $\phi$ several times without further reference.


The next lemma is the properties of the perfect code subgroup of a finite group and we use it in the main result. 
\begin{lemma}\label{1.2}(see 
	\cite[Theorem 1.2]{1} and \cite[Lemma 2.2]{9}) Let $G$ be a group and $H$ a subgroup of $G$. Then the following are equivalent:
	
	(a) $H$ is a perfect code of $G$;
	
	(b) there exists an inverse-closed right transversal of $H$ in $G$;
	
	(c) for each $x \in G$ such that $x^2 \in H$ and $| H| / | H \cap H^x| $ is odd, there exists $y \in Hx$ such that $y^2=1$;
	
	(d) for each $x \in G$ such that $HxH = Hx^{-1}H$ and $|H|/|H \cap H^x|$ is odd, there exists $y \in Hx$ such that $y^2 = 1$.
\end{lemma} 

\begin{lemma}\label{1}
	 Let $G$ be a group, $H$ a subgroup of $G$ and $x\in G\setminus H$. Then,  for each $w\in HxH$,  $\left| xH \cap Hx\right| =\left| H \cap H^x\right|=|H\cap H^w|=\left| wH \cap Hw\right|$. Moreover, the numbers of involutions in $Hx$ and $Hxh$ are equal, for each $h\in H$.  
\end{lemma}
\begin{proof}	Let $T=xH\cap Hx$. We note that,   for $h\in H$, $xh\in T$  if and only if $xhx^{-1}\in H$.  Therefore, $T=\{xh\mid h\in H\cap H^x\}$ and so $\left| T\right| =\left| H\cap H^x\right| $. Thus, by the same manner if $w=kxh$  for some  $h,k\in H$, then  $|wH\cap Hw|=\left|kxhH\cap Hkxh\right| =\left| H\cap H^{kxh}\right|=|(H\cap H^x)^h| =\left| H\cap H^x\right| $.  Now, we show that the numbers of involutions in  $Hx$ and $Hxh$ are the same. If $a\in H$ and $ax$ is an involution in $Hx$, then $(ax)^h$ is an involution in $Hxh$. Let $b\in H$ and $bxh$ be an involution in $Hxh$. Then,  $bxh=h^{-1}(hbx)h=(hbx)^h$. Hence, $hbx$ is an involution in $Hx$. Therefore, the numbers of involutions in these two cosets are equal.   
\end{proof}
\begin{lemma}\label{key2}
	Let $H$ be a subgroup of a group $G$ and $x\in G\setminus H$.  Let  $v\in HxH$ and $y\in Hx^{-1}H$. Then $Hv$ contains exactly $|H\cap H^x|$ elements whose inverses belong to $Hy$, i.e., $|(Hv)^{-1}\cap Hy|=|H\cap H^x|$.  In particular, if $HxH=Hx^{-1}H$, then $Hv$ contains exactly $|H\cap H^x|$ elements whose inverses belong to $Hv$. 
\end{lemma}
\begin{proof}
	Let $m=|H\cap H^x|$.  Suppose  that $v=h'xh$, $y=k'x^{-1}k$, for some $h,h',k,k'\in H$. Let $z=k^{-1}xh$. Then $z \in Hv$ and $z^{-1}\in Hy$. Hence, we have $Hv=Hz$ and $Hy=Hz^{-1}$ and so  $|(Hv)^{-1}\cap Hy|=|\left( Hz\right) ^{-1}\cap
	Hz^{-1}|=|z^{-1}H\cap Hz^{-1}|=|H\cap H^{z^{-1}}|=|H\cap H^{x^{-1}}|=m$, by Lemma \ref{1}. 
\end{proof}
\begin{lemma}\label{graph}
	Let $G$ be a group and $H$ a subgroup of $G$. Let $x\in G\setminus H$,	
	  $m=|H\cap H^x|$ and $t=|H|/m$.   Then $\Gamma_x$ has $m$ simple  subgraphs $\Gamma_i$, $i=1,\dots, m$, where	 
	   ${\bf V}(\Gamma_i)={\bf V}(\Gamma_x)$	   
	    for $i\in \{1,\dots, m\}$,    ${\bf E}(\Gamma_i)\cap {\bf E}(\Gamma_j)=\emptyset$, for  $i\not =j$, and    
   ${\bf E}(\Gamma_x)=\bigcup\limits_{i=1}^m {\bf E}(\Gamma_i)$.   Moreover, if  $HxH\not =Hx^{-1}H$, then  each $\Gamma_i$ is isomorphic to $K_{t,t}$ and if  $HxH = Hx^{-1}H$, then  each $\Gamma_i$ is isomorphic to $ K_{t}$,  for  $1\leq i\leq m$. 
\end{lemma}
\begin{proof}
From Lemma \ref{key2}, between each pair of  distinct   vertices $Hy$ and $Hz$  in $ {\bf V}(\Gamma_x)$, where $y\in HxH$ and $z\in Hx^{-1}H$  there exist exactly $m$   distinct edges  $\mathfrak{E}_{\{y,z\}}=\{\mathfrak{e}_{\{y,z\}}^i\ | \ i=1,\dots , m\}$.  Let  $T_1$ and $T_2$  be  arbitrary right transversals of $H$ in $HxH$ and $ Hx^{-1}H$, respectively. In case $HxH=Hx^{-1}H$, we take $T_1=T_2$.  Let 
$\mathcal{T}=\{\{y,z\}\mid \ y\in T_1, z\in T_2 \  \text{and}\ y\not =z \}$. 
Then for each $i\in \{1,\dots,m\}$, we take the subgraph $\Gamma_i$ to be a  simple graph  with vertex set $ {\bf V}(\Gamma_i)={\bf V}(\Gamma_x)$ and the edge set
${\bf E}(\Gamma_i)=\{\mathfrak{e}^i_{\left\lbrace y,z\right\rbrace } |\ \ {\left\lbrace y,z\right\rbrace }\in \mathcal{T}\}$.  Clearly,   ${\bf E}(\Gamma_i)\cap {\bf E}(\Gamma_j)=\emptyset$  for  $i\not =j$, and    
   ${\bf E}(\Gamma_x)=\bigcup\limits_{i=1}^m {\bf E}(\Gamma_i)$, as claimed.   Remark that   $t=|H|/m=|HxH|/|H|=|Hx^{-1}H|/|H|$,  which means that $HxH$ and $Hx^{-1}H$ are the  union of exactly $t$ distinct right cosests of $H$.   If $HxH\not =Hx^{-1}H$, then   $\Gamma_i\cong K_{t,t}$ and  if $HxH =Hx^{-1}H$, we have $\Gamma_i\cong K_t$, by Lemma \ref{key2}. 
\end{proof}

\begin{figure}[h]
\begin {center}
\begin {tikzpicture}[auto ,node distance =3cm  ,on grid ,
semithick ,
state/.style ={ circle ,top color =white , bottom color = processred!20 ,
draw,processred , text=red , minimum width =0 cm}]
\node[state] (C){$Hx$};
\node[state] (A) [below left=of C] {$Hy$};
\node[state] (B) [below  right =of C] {$Hz$};
\path (C) edge [bend left =15] node[below =0.15 cm] {} (A);
\path (A) edge [bend right = -15,red] node[below =0.15 cm] {} (C);
\path (A) edge [bend left =15] node[above] {} (B);
\path (B) edge [bend left =15,red] node[below =0.15 cm] {} (A);
\path (C) edge [bend left =15,red] node[below =0.15 cm] {} (B);
\path (B) edge [bend right = -15] node[below =0.15 cm] {} (C);

\end{tikzpicture}
\hspace{2cm}
\begin{tikzpicture}[circ/.style={circle, draw, fill}]
	\vertex[fill] (v1) at (0,0) [label=above:$x_{1}$] {};
	\vertex[fill] (v2) at (.5,0) [label=above:$x_{2}$] {};
	\vertex[fill] (v3) at (1,0) [label=above:$x_{3}$] {};
	\vertex[fill] (v4) at (1.5,0) [label=above:$x_{4}$] {};
	\vertex[fill] (v5) at (2,0) [label=above:$x_{5}$] {};
	\vertex[fill] (v6) at (2.5,0) [label=above:$x_{6}$] {};
	
	\vertex[fill] (v7) at (-1,-1) [label=left:$y_{1}$] {};
	\vertex[fill] (v8) at (-1,-1.5) [label=left:$y_{2}$] {};
	\vertex[fill] (v9) at (-1,-2) [label=left:$y_{3}$] {};
	\vertex[fill] (v10) at (-1,-2.5) [label=left:$y_{4}$] {};
	\vertex[fill] (v11) at (-1,-3) [label=left:$y_{5}$] {};
	\vertex[fill] (v12) at (-1,-3.5) [label=left:$y_{6}$] {};
	
	\vertex[fill] (v13) at (3.5,-1) [label=right:$z_{1}$] {};
	\vertex[fill] (v14) at (3.5,-1.5) [label=right:$z_{2}$] {};
	\vertex[fill] (v15) at (3.5,-2) [label=right:$z_{3}$] {};
	\vertex[fill] (v16) at (3.5,-2.5) [label=right:$z_{4}$] {};
	\vertex[fill] (v17) at (3.5,-3) [label=right:$z_{5}$] {};
	\vertex[fill] (v18) at (3.5,-3.5) [label=right:$z_{6}$] {};
	
	\path	(v1) edge [red] (v7);
	\path	(v2) edge  (v8);
	\path	(v6) edge [red] (v13);
	\path	(v5) edge  (v14);
	\path   (v12) edge [red] (v18);
	\path  (v11) edge  (v17);
	

\end{tikzpicture}
 \caption{{\bf (a)} $\Gamma_x$  \hspace{6cm}  {\bf (b)} $\Delta_x$\hspace{3cm}}
\end{center} 
\end{figure}
\begin{figure}
	\begin{center}

		\begin {tikzpicture}[auto ,node distance =3cm  ,on grid ,
		semithick ,
		state/.style ={ circle ,top color =white , bottom color = processred!20 ,
			draw,processred , text=red , minimum width =0 cm}]

		\node[state](DD)[right =of B]{$Hw$};
		\node[state] (CC) [above= of DD]{$Hy$};
		\node[state] (BB) [right  =of CC] {$Hx$};
		\node[state] (AA) [below=of BB] {$Hz$};
		\path (AA) edge [bend left =15] node[above] {} (BB);
		\path (BB) edge [bend left =15,red ] node[below =0.15 cm] {} (AA);
		\path (BB) edge [bend left =15,red] node[right=0.15 cm] {} (DD);
		\path (DD) edge [bend right = -15] node[left =0.15 cm] {} (BB);
		\path (CC) edge [bend left =15,red] node[left =0.15 cm] {} (AA);
		\path (AA) edge [bend right = -15,] node[right=0.15 cm] {} (CC);
		\path (CC) edge [bend left =15] node[below] {} (DD);
		\path (DD) edge [bend left =15, red] node[above =0.15 cm] {}(CC);
	\end{tikzpicture}
\hspace{3cm}
\begin{tikzpicture}[circ/.style={circle, draw, fill}]
			
			\vertex[fill] (w1) at (4,0) [label=above:$y_{4}$] {};
			\vertex[fill] (w2) at (4.5,0) [label=above:$y_{3}$] {};
			\vertex[fill] (w3) at (5,0) [label=above:$y_{2}$] {};
			\vertex[fill] (w4) at (5.5,0) [label=above:$y_{1}$] {};
			\vertex[fill] (w5) at (7,0) [label=above:$x_{4}$] {};
			\vertex[fill] (w6) at (7.5,0) [label=above:$x_{3}$] {};
			\vertex[fill] (w7) at (8,0) [label=above:$x_{2}$] {};
			\vertex[fill] (w8) at (8.5,0) [label=above:$x_{1}$] {};
			\vertex[fill] (w9) at (4,-3.5) [label=below:$w_{4}$] {};
			\vertex[fill] (w10) at (4.5,-3.5) [label=below:$w_{3}$] {};
			\vertex[fill] (w11) at (5,-3.5) [label=below:$w_{2}$] {};
			\vertex[fill] (w12) at (5.5,-3.5) [label=below:$w_{1}$] {};
			\vertex[fill] (w13) at (7,-3.5) [label=below:$z_{4}$] {};
			\vertex[fill] (w14) at (7.5,-3.5) [label=below:$z_{3}$] {};
			\vertex[fill] (w15) at (8,-3.5) [label=below:$z_{2}$] {};
			\vertex[fill] (w16) at (8.5,-3.5) [label=below:$z_{1}$] {};
			
			\path	(w1) edge  [red](w9);
			\path	(w2) edge (w10);
			\path	(w3) edge (w13);
			\path	(w4) edge [red] (w14);
			\path   (w11) edge  (w5);
			\path  (w12) edge [red] (w6);
			\path  (w7) edge (w15);
			\path  (w8) edge [red] (w16);	
			
		\end{tikzpicture}
	 \caption{{\bf (a)} $\Gamma_x$  \hspace{6cm}  {\bf (b)} $\Delta_x$}
\end{center}
\end{figure}

\begin{example}\label{ex} To make  Lemma  \ref{graph} more clear, let $H$ be a subgroup of $G$ and $x\in G\setminus H$ such that $HxH=Hx^{-1}H$, $t=|H|/|H\cap H^x|=3$ and $m=|H\cap H^x|=2$.  Then by Lemma \ref{key2},  Figure 1(a) and Figure 1(b) are $\Gamma_x$ and $\Delta_x$, respectively. By the notations in Lemma \ref{graph},  ${\bf V}(\Gamma_1)={\bf V}(\Gamma_2)={\bf V}(\Gamma_x)$ and assume the red  edges  and the black edges of $\Gamma_x$ are the  edges of $\Gamma_1$ and $\Gamma_2$, respectively.  Then,  we may assume the bijective function $\phi$ (described in the first part of this section)  maps the edges of  $\Gamma_1$  (resp. $\Gamma_2$) to the red (resp. black) edges of $\Delta_x$. Thus,  ${\bf V}(\phi ({\bf E}(\Gamma_1)))=\{x_1, y_1=x_1^{-1}, x_6, z_1=x_6^{-1}, y_6,z_6=y_6^{-1}\}$ and ${\bf V}(\phi ({\bf E}(\Gamma_2)))=\{x_2,y_2=x_2^{-1}, x_5, z_2=x_5^{-1}, y_5,z_5=y_5^{-1}\}$.  
\end{example}

\begin{example}\label{ex2}
If $H$ is   a subgroup of $G$ and $x\in G\setminus H$ such that $HxH\not=Hx^{-1}H$, $|H|/|H\cap H^x|=2$ and $|H\cap H^x|=2$, then Figure 2(a) and Figure 2(b) are $\Gamma_x$ and $\Delta_x$, respectively. By the notations of Lemma \ref{graph}, ${\bf V}(\Gamma_1)={\bf V}(\Gamma_2)={\bf V}(\Gamma_x)$ and assume the red edges  and the black edges of $\Gamma_x$ are the  edges of $\Gamma_1$ and $\Gamma_2$, respectively.  Then, we may assume the bijective function $\phi$  maps the edges of  $\Gamma_1$  (resp. $\Gamma_2$) to the red (resp. black) edges of $\Delta_x$. Then we have ${\bf V}(\phi ({\bf E}(\Gamma_1)))=\{x_1,z_1=x_1^{-1},x_3,w_1=x_3^{-1},y_1,z_3=y_1^{-1},y_4,w_4=y_4^{-1}\}$ and  ${\bf V}(\phi ({\bf E}(\Gamma_2)))=\{x_2,z_2=x_2^{-1},x_4,w_2=x_4^{-1}, y_2,z_4=y_2^{-1},y_3, w_3=y_3^{-1}\}$.
\end{example}
\begin{remark}\label{rem} Let  $\phi: {\bf E}(\Gamma_x)\rightarrow {\bf E}(\Delta_x)$ be the bijection defined in the first part of this section. If  $L_1 , L_2\subseteq {\bf E}(\Gamma_x)$ such that  $L_1\cap L_2=\emptyset$, then  $\phi(L_1)\cap \phi(L_2)=\emptyset$, since $\phi$ is one to one.   Note that   $\Delta_x$ is a simple graph such that the degree of each vertex is either $0$ or $1$. 
Therefore,   ${\bf V}(\phi (L_1))\cap{\bf V}(\phi(L_2))=\emptyset$. 
\end{remark}

\begin{lemma}\label{t1}
	Let $G$ be a group and $H$ a subgroup of $G$. Let $x\in G\setminus H$ such that $HxH\ne Hx^{-1}H$. Then $H$ has $|H|$ pairwise disjoint inverse-closed right transversals in $HxH\cup Hx^{-1}H$. 
\end{lemma}
\begin{proof}
	Let  $m=|H\cap H^x|$ and $t=\left| H\right| / m $.
	  Using the notations in Lemma \ref{graph}, for each $1\leq i\leq m$,  $\Gamma_i\cong K_{t,t}$, and so  $\Gamma_i$    has $t$ disjoint perfect matchings $\mathfrak{M}_{i1}, \dots, \mathfrak{M}_{it}$, by K$\ddot{o}$nig’s 1-factorization theorem \cite{konig}.   Then,   for each $\mathfrak{M}_{ij}$ we get that   ${\bf V}(\mathfrak{M}_{ij})$ is the set of all right cosets of $H$ in $HxH\cup Hx^{-1}H$ and thus ${\bf V}(\phi(\mathfrak{M}_{ij}))$ is  a right transversal of $H$ in $HxH\cup Hx^{-1}H$. By the definition of $\Delta_x$,  the ends of each edge in $\Delta_x $ are the inverses of each other. Hence, ${\bf V}(\phi(\mathfrak{M}_{ij}))$ is   an inverse-closed  right transversal of $H$ in $HxH\cup Hx^{-1}H$, for each $i$ and $j$.  Now, we claim  that ${\bf V}(\phi(\mathfrak{M}_{ij}))\cap {\bf V}(\phi (\mathfrak{M}_{i'j'}))=\emptyset$, when $(i,j)\not=(i',j')$. As  $\mathfrak{M}_{i1}, \dots, \mathfrak{M}_{it}$,  are disjoint perfect matchings, then  by Remark \ref{rem}, ${\bf V}(\phi(\mathfrak{M}_{ij}))\cap {\bf V}(\phi (\mathfrak{M}_{ij'}))=\emptyset$, for $j\not=j'$.   By Lemma \ref{graph},  ${\bf E}(\Gamma_k)\cap {\bf E}(\Gamma_l)=\emptyset$ and so by  Remark \ref{rem},  ${\bf V}(\phi({\bf E}(\Gamma_k)))\cap  {\bf V}(\phi({\bf E}(\Gamma_l)))=\emptyset$, for each $1\leq k<l\leq m$. Thus, ${\bf V}(\phi(\mathfrak{M}_{ij}))\cap {\bf V}(\phi (\mathfrak{M}_{i'j'}))=\emptyset$, when $(i,j)\not=(i',j')$, as we claimed. Therefore,   we have exactly $|H|=mt$ disjoint  inverse-closed right transversals of $H$ in $HxH\cup Hx^{-1}H$, as wanted.       
\end{proof}
\begin{lemma}\label{t}
		Let $G$ be a group and $H$ a subgroup of $G$. Let $x\in G\setminus H$ such that $HxH= Hx^{-1}H$. If $|H|/|H\cap H^x|=2n$, for some integer $n$, then $H$ has $|H|-|H\cap H^x|$ pairwise disjoint inverse-closed right transversals in $HxH$. 
\end{lemma}
\begin{proof}
	 Let $m=|H\cap H^x|$. Using the same notations in Lemma \ref{graph}, for $1\leq i\leq m$, $\Gamma_{i}\cong K_{2n}$and so $\Gamma_i$ has exactly $2n-1$  disjoint perfect matchings  $\mathfrak{M}_{ij}$, for $j=1,\dots,2n-1$. If $\mathfrak{M}_{ij}$ is a perfect matching of $\Gamma_i$, then {\bf V}$(\phi(\mathfrak{M}_{ij}))$ is  an inverse-closed right transversal of $H$  in $HxH$.    So, by Remark \ref{rem} and  similarly to the proof of  Lemma \ref{t1}, we can find $\left( 2n-1\right) m=2nm-m=|H|-|H\cap H^x|$ disjoint inverse-closed right transversals of $H$ in $HxH$.
\end{proof}
\begin{lemma}\label{t2}
	Let $G$ be a group and $H$ a subgroup of $G$. Let $x\in G\setminus H$ such that $HxH= Hx^{-1}H$ and $|H|/|H\cap H^x|$ is even. Then for each $0\le b\le |H|$, there exist $b$ pairwise disjoint right transversals of $H$ in $HxH$  whose union is inverse-closed. 
\end{lemma}
\begin{proof}
Let $m=|H\cap H^x|$ and $t=|H|/m$.	Suppose that $T_i$,   $1\le i\le |H|-m$,  are disjoint  inverse-closed   right transversals of $H$ in $HxH$, as described in Lemma \ref{t}. If $b\le |H|-m$, then clearly $\bigcup\limits_{i=1}^b T_i$ is inverse-closed and we get the result.  So, suppose that  $b> |H|-m$. Let $M=HxH\setminus \bigcup_{i=1}^{|H|-m}T_i$.  Since  $HxH=\bigcup\limits_{h\in H}Hxh$, there exist some $h_j\in H$,  $j=1,\dots, t$, such that  $HxH=\bigcup\limits_{j=1}^{t}Hxh_j$,  a disjoint  union of  cosets of $H$ in $HxH$. Then  $|Hxh_j\cap T_i|=1$,  for each $i=1,\dots,|H|-m$ and $j=1,\dots,t$. Thus,   for each $j$, $Hxh_j\setminus \bigcup_{i=1}^{|H|-m}T_i$  has exactly $m$ elements.    Hence, we conclude that  $M=HxH\setminus \bigcup_{i=1}^{|H|-m}T_i$ is a union of $m$  disjoint right transversals of $H$ in $HxH$.  As $HxH=Hx^{-1}H$ and  $\bigcup\limits_{i=1}^{|H|-m} T_i$ is inverse-closed, it follows that  $M=M^{-1}$.  Remark that $|H|/m\geq 2$, we have  $b> |H|-m\geq m$. Then,  $M\cup \bigcup\limits_{i=1}^{b-m}  T_i$ is  a union of $b$ right transversals of $H$ in $HxH$ and it is inverse-closed. 
\end{proof}
\begin{lemma}\label{odd}
Let $\Gamma\cong K_{2n+1}$,  for some integer $n\geq 1$,  and ${\bf V}(\Gamma)=\{v_1,\dots,v_{2n+1}\}$. Then, for each $1\leq i\leq 2n+1$, there exists   a matching $E_i$ for $\Gamma$ with vertex set ${\bf V}(\Gamma)\setminus\{v_i\}$, such that  for $i\not =j$,  $E_i\cap E_j=\emptyset $. 
\end{lemma}
\begin{proof}
 For each $a\in \mathbb{Z}$, assume  $\overline{a}$ is an integer such that  $1\leq \overline a\leq 2n+1$ and  $ a\equiv\overline{a} \pmod{2n+1} $.  Then for each $i=1,\dots, 2n+1 $, we define the matching with $n$ edges as following (see Figure 3): 
 $$E_{i}=\{\{v_{\overline{i-l}},  v_{\overline{i+l}}\}\in {\bf E}(\Gamma) \mid 1\leq l\leq n   \}.$$ To prove that $E_i$ is a matching we need to show that every two edges in $E_i$ have disjoint ends. On the contrary, assume that $E_i$ has two distinct  edges which have an end in common.  Then, there is $1\leq l<k\leq n$ such that $\overline {i+k}\equiv  \overline{i\pm l}\pmod{2n+1}$. Therefore,  $k\mp l\equiv 0 \pmod{2n+1} $,   which is not possible, as  $1\leq l<k\leq n$. Thus,  $E_i$ is a matching whose vertex set is ${\bf V}(\Gamma)\setminus \{v_i\}$.  Now, we prove that for $1\leq i<j\leq 2n+1$,   $E_i$ and $E_j$ are  disjoint matchings. On the contrary, assume that $e\in E_i\cap E_j$. Thus, there exist $1\leq l\leq k\leq 2n+1 $ such that   $e=\{v_{\overline{i-l}}, v_{\overline{i+l}}\}= \{v_{\overline{j-k}},  v_{\overline{j+k}}\}$.  So either $i+l\equiv j+k \pmod{2n+1}$ and $i-l\equiv j-k\pmod{2n+1}$; or $i+l\equiv j-k \pmod{2n+1}$ and $i-l\equiv j+k\pmod{2n+1}$.  In both cases $2i\equiv 2j \pmod {2n+1}$, which implies that $i\equiv j \pmod{2n+1}$ and so $i=j$, a contradiction.   
\end{proof}
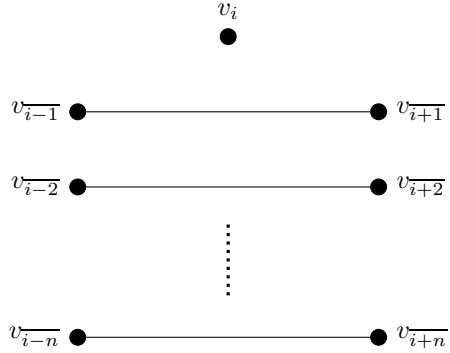
\begin{figure}
\begin {center}

	\begin{tikzpicture}
	\vertex[fill] (v1) at (0,-2) [label=above:$v_{i}$] {};
		
		\vertex[fill] (v2) at (-2,-3) [label=left:$v_{\overline{i-1}}$] {};
		\vertex[fill] (v3) at (2,-3) [label=right:$v_{\overline{i+1}}$] {};
		\vertex[fill] (v4) at (-2,-4) [label=left:$v_{\overline{i-2}}$] {};
		\vertex[fill] (v5) at (2,-4) [label=right:$v_{\overline{i+2}}$] {};
		
		\vertex[fill] (v6) at (-2,-6) [label=left:$v_{\overline{i-n}}$] {};
		\vertex[fill] (v7) at (2,-6) [label=right:$v_{\overline{i+n}}$] {};
	
		\path	(v2) edge  (v3);
		\path	(v4) edge  (v5);
		\path	(v6) edge  (v7);
		\path	(0,-4.5) edge [dotted, very thick]  (0,-5.5);

\end{tikzpicture}
\end{center}
\caption{A maximal  matching in  a complete graph with odd vertices}
\end{figure}
	\begin{theorem}\label{t3}	Let $G$ be a group and $H$ a subgroup of $G$. Let $x\in G\setminus H$ such that $HxH= Hx^{-1}H$ and $|H|/|H\cap H^x|=2n+1$, for some integer $n$. 
		
		$1)$ If $b$ is an even integer with   $0\le b\le |H|$, then there exist  $b$ pairwise disjoint right transversals of $H$ in $HxH$ whose union is inverse-closed.
		
		$2)$ If $H$ is a perfect code of $G$, then for  each $0\le b\le |H|$, there exist $b$ pairwise disjoint right transversals of $H$ in $HxH$, whose union is inverse-closed.
\end{theorem}
\begin{proof}
	Let $m=|H\cap H^x|$.  By Lemma \ref{key2}, we get that $Hx$ has a maximal inverse-closed subset  $A$ with $m$ elements. If  $Hx$ and so $A$ has exactly   $c$ involutions, for some integer $c\geq 0$, then $|A|-c=2d$, for some $d\ge 0$, which means $m=c+2d$.  Let  $HxH=Hx_1\cup Hx_2\cup\dots\cup Hx_{2n+1}$, for some  $x_i\in HxH$, $1\le i\le 2n+1$.  By Lemmas \ref{1} and  \ref{key2}, for each  $i=1,\dots, 2n+1$,  $Hx_i$  has  a maximal inverse-closed subset with $m $ elements,  and this  set contains exactly $c$ involutions, say  $\{r_{i1},\dots, r_{ic}\}$.  So  we  set $\left\lbrace r_{i1},r_{i2},\dots,r_{ic},s_{i1},s_{i1}^{-1},s_{i2},s_{i2}^{-1},\dots,s_{id},s_{id}^{-1}\right\rbrace $  to be the  inverse-closed subset of   $Hx_i$.  Using the same notations in Lemma \ref{graph}, we see that $\Gamma_x$ has $m$ subgraphs $\Gamma_j\cong K_{2n+1}$,  where $1\le j\le m$.  We remind that ${\bf E}(\Gamma_x)=\bigcup\limits_{j=1}^m {\bf E}(\Gamma_j)$ and ${\bf E}(\Gamma_j)\cap {\bf E}(\Gamma_{j'})=\emptyset$,   for each $1\leq j<j'\leq m$ and by Remark \ref{rem}, we get that  {\bf V}$(\phi({\bf E}(\Gamma_j)))\cap${\bf V}$(\phi({\bf E}(\Gamma_{j'})))=\emptyset$.  By Lemma \ref{odd},  $\Gamma_j$ has $2n+1$ pairwise disjoint matchings $E_{1,j},\dots, E_{2n+1,j}$,  for each $1\leq j\leq m$, such that ${\bf V}(E_{i,j})={\bf  V}(\Gamma_j)\setminus \{Hx_i\}$. 
  Thus,   {\bf V}$(\phi(E_{i,j}))$ is an inverse-closed right transversal of $H$ in $HxH\setminus Hx_i$.  
	For $1\leq i\leq 2n+1$ and  $1\leq j\leq c$,  set $T_{i,j}=\{r_{ij}\}\cup {\bf V}(\phi(E_{i,j}))$.   Then,  $T_{i,j}$ forms an inverse-closed right transversal of  $H$ in $HxH$, as $r_{ij}$ is an involution in $Hx_i$. 	
	 Also for $1\leq i\leq 2n+1$ and  $1\leq l\leq d$, 
	 we set  $R_{i, l}=\{s_{il}, s_{il}^{-1}\}\cup {\bf V}(\phi(E_{i,c+2l-1}))\cup {\bf V}(\phi( E_{i,c+2l}))$, which   is a  union of two disjoint right transversals of $H$ in $HxH$, and clearly,  is inverse-closed. 
	
	 Set $\mathcal{T}=\{T_{i,j} \mid \  1\leq i\leq 2n+1 \  \text{and} \ 1\leq j \leq c \}$ and $\mathcal{R}=\{R_{i,l} \mid \ 1\leq i\leq 2n+1 \  \text{and} \  1\leq l\leq d \}$. If  $S$  is a member of  $\mathcal{T}$ or $\mathcal{R}$, we say that $S$ is of type $\mathcal{T}$ or $\mathcal{R}$, respectively. We claim that members of $\mathcal{T}\cup \mathcal{R}$ are pairwise disjoint. We assume that  there exist two distinct members  $S_1, S_2\in \mathcal{T}\cup \mathcal{R}$  such that $S_1\cap S_2\not = \emptyset$.  To get a contradiction, we consider the following cases, separately.   
	 
	$\bullet$ First,  let both sets be of type $\mathcal{T}$. Then $S_1=T_{i,j}=\{r_{ij}\}\cup{\bf V}(\phi(E_{i,j}))$ and $S_2=T_{i',j'}=\{r_{i'j'}\}\cup {\bf V}(\phi(E_{i',j'}))$, for some integers $1\leq i,i'\leq 2n+1$ and  $1\leq j, j'\leq c$. Note that  by Lemma \ref{graph},  $E_{i,j}\cap E_{i',j'} \subseteq {\bf E}(\Gamma_{j})\cap {\bf E}(\Gamma_{j'})=\emptyset$,  for $j\not =j'$ and by Lemma  \ref{odd},  $E_{i,j}\cap E_{i',j}=\emptyset$,  for $i\not =i'$. Thus,  if  $(i,j)\not =(i',j')$,  then $E_{i,j}$  and $E_{i',j'}$  are  disjoint subsets of $\Gamma_x$  and so by Remark \ref{rem},  ${\bf V}(\phi(E_{i,j}))\cap {\bf V}(\phi(E_{i',j'}))=\emptyset$.  Hence,
 $S_1$ and $ S_2$ contain $r_{ij}=r_{i'j'}$, which implies that $i=i'$ and $j=j'$, a contradiction. 

	 $\bullet$ Now, let  $S_1$ and $S_2$ be of type $\mathcal{R}$. Then,   $S_1=R_{i,l}=\{s_{il}, s_{il}^{-1}\}\cup {\bf V}(\phi(E_{i,c+2l-1}))\cup {\bf V}(\phi(E_{i,c+2l}))$ and $S_2=R_{i',l'}=\{s_{i'l'}, s_{i'l'}^{-1}\}\cup {\bf V}(\phi(E_{i',c+2l'-1}))\cup {\bf V}(\phi(E_{i',c+2l'}))$, for some integers $1\leq i,i'\leq 2n+1$ and $1\leq l, l'\leq d$.   Note that  by Lemma \ref{graph},  $(E_{i,c+2l-1}\cup E_{i,c+2l}) \cap (E_{i',c+2l'-1}\cup E_{i', c+2l'}) \subseteq ({\bf E}(\Gamma_{c+2l-1})\cup {\bf E}(\Gamma_{c+2l}))\cap ({\bf E}(\Gamma_{c+2l'-1})\cup{\bf E}(\Gamma_{c+2l'}))=\emptyset$,  for $l\not =l'$ and by Lemma \ref{odd},  $(E_{i,c+2l-1}\cup E_{i,c+2l}) \cap (E_{i',c+2l-1}\cup E_{i', c+2l})=\emptyset$,  for $i\not =i'$. Thus, if  $(i,l)\not =(i',l')$, then $(E_{i,c+2l-1}\cup E_{i,c+2l})\cap(E_{i',c+2l'-1}\cup E_{i',c+2l'})=\emptyset$   and so by Remark \ref{rem},  $({\bf V}(\phi(E_{i,c+2l-1}))\cup {\bf V}(\phi(E_{i,c+2l})))\cap ({\bf V}(\phi(E_{i',c+2l'-1}))\cup {\bf V}(\phi(E_{i',c+2l'})))=\emptyset$.
  Thus, $S_1\cap  S_2$ contains $\{s_{il}, s_{il}^{-1}\}=\{s_{i'l'}, s_{i'l'}^{-1}\}$, which means that  $i=i'$ and $l=l'$,  a contradiction.  
	 
	 $\bullet$ Finally, let  $S_1$ and $S_2$ be  of type $\mathcal{T}$ and  type $\mathcal{R}$, respectively.  Let $S_1=T_{i,j}$ and $S_2=R_{i',l}$ for some	 
	  $1\leq i,i'\leq 2n+1$, $1\leq j\leq c$ and $1\leq l\leq d$.  Note that  $ S_2$ does not contain any involution.  
	   On the other hand,  $S_1\setminus \{r_{ij}\}\subseteq  {\bf V}(\phi({\bf E}(\Gamma_{j})))$ and   $S_2\setminus \{s_{i'l}, s^{-1}_{i'l}\}\subseteq  {\bf V}(\phi({\bf E}(\Gamma_{c+2l-1})))\cup   {\bf V}(\phi({\bf E}(\Gamma_{c+2l})))$ and as they are disjoint by Remark \ref{rem}, we get a contradiction. So our claim is proved. 
	 
	We first aim to prove Part (1) of this theorem.  Assume that $b$ is even. 	
	 If $b\leq 2d(2n+1)\le |H|=m(2n+1)$,  then we take a  union of  $b/2$ distinct  sets of type  $\mathcal{R}$. Clearly,  this set is inverse-closed, as requested.  If $b>  2d(2n+1)$, then  $\bigcup_{j=1}^d\bigcup_{i=1}^{2n+1} R_{i,j}$  with  $b-2d(2n+1)$ distinct  sets of type  $\mathcal{T}$ gives us the desired set.
	 So Part (1) is proved. 
	 
	  Now, we prove Part (2) of the theorem.  By Part (1), we suppose that  $b$ is odd.    Note that, as  $H$ is a subgroup perfect code,  we conclude that $c\geq 1$ and so   $\mathcal{T}$ is not empty.  If   $b<2d(2n+1)$, we  take $(b-1)/2$ sets of type  $\mathcal{R}$ and the union of them with  $T_{1,1}$ forms an  inverse-closed set of $b$ pairwise disjoint right transversals of $H$. Remind  that $b\leq |H|=m(2n+1)$.  If $b>  2d(2n+1)$, then similarly to the previous case, $\bigcup_{j=1}^d\bigcup_{i=1}^{2n+1} R_{i,j}$  with  $b-2d(2n+1)$ distinct  sets of type  $\mathcal{T}$ gives us the desired set. So the proof of  Part (2) is complete. 
\end{proof}
Now, we are ready to prove Theorem A.
\begin{PMa} The "only if" part is clear. So we prove the "if" part.  Let $H$ be a perfect code of $G$. By Lemmas \ref{t1}, \ref{t2} and Theorem \ref{t3}(2),  for each $x\in G\setminus H$, and for each $0\leq b \leq |H|$, there exists an inverse-closed  set of union of $b$ pairwise  disjoint  right transversals of $H$ in $HxH\cup Hx^{-1}H$, say $T_x^b$.  Let  $G=\bigcup\limits_{k=1}^{\gamma} (Hx_kH\cup Hx_k^{-1}H)$ be a disjoint union of double cosets of $H$ in $G$,  for some $x_k\in G$, $1\le k\le \gamma$. Then,   $T^b=\bigcup\limits_{k=1}^{\gamma} T_{x_k}^b$  construct  an inverse-closed  set of union of $b$ pairwise disjoint  right transversals of $H$ in $G$. 
	
	By assumption,   ${\rm gcd}(2, |H| -1)$ divides $a$. Then, 
	$a$ is even, if $|H|$ is odd. If $|H|$ is odd, then $H \setminus \{1\}$ is partitioned into pairs of elements that are
	inverses of each other, and so $H \setminus \{1\}$ has an inverse-closed subset of size $a$ for each
	even integer $0\leq a\leq  |H|-1$.  Let $|H|$ be even  and $I$ be the set of involutions in $H $. Clearly, $|I|$ is odd. If $a\le|I|$, then we take $T^a\subseteq I$ to be a set of size $a$. If $a> |I|$ and $a$  is even, then we take $T^a$ to be  a union of $|I|-1$ elements of $I$ with  an inverse-closed  subset of  $H\setminus \left( I\cup \left\lbrace 1\right\rbrace \right) $  of size $a-|I|+1$.  If $a> |I|$ and $a$  is odd, then take $T^a$ to be a  union of $I$ with an inverse-closed  subset of $H\setminus \left( I\cup \left\lbrace 1\right\rbrace \right) $  with $a-|I|$ elements.  Thus,  in any case, we conclude that
	there exists an inverse-closed subset $T^a$ of $H\setminus \{1\}$ with $|T^a| = a$. 	
	Then,  setting   $S=T^b\cup T^a$,  $H$ is an $(a,b)$-regular set of $\Cay(G, S)$, as desired.  \qed
\end{PMa}
The proof of Theorem B is also  a consequence  of   Lemmas \ref{t1}, \ref{t2} and Theorem \ref{t3}:
\begin{PMb}
	Let $b$ be any even integer. Then using  Lemmas \ref{t1}, \ref{t2} and Theorem \ref{t3}(1),  for each $x\in G\setminus H$, and for every even integer  $0\leq b \leq |H|$, there exists an inverse-closed  set which is a union of $b$ pairwise  disjoint  right transversals of $H$ in $HxH\cup Hx^{-1}H$, say $T_x^b$. Hence,  similarly  to the  proof of  Theorem A we get the results. \qed
\end{PMb}


\begin{thebibliography}{1}		
\bibitem{bcg} R. A. Bailey, P. J. Cameron, A. L. Gavrilyuk and S. V. Goryainov, Equitable partitions of
Latin-square graphs, J. Combin. Des., 27 (3) (2019) 142-160.	
\bibitem{hamming2coloring1}E. A. Bespalov, D. S. Krotov, A. A. Matiushev, K. V. Vorob’ev, Perfect 2-colorings of Hamming graphs, J. Combin. Des.  29 (6) (2021) 1-30.
\bibitem{R1}D. M. Cardoso, An overview of $( \kappa, \tau )$-regular sets and their applications, Discrete Appl. Math. 269 (2019) 2-10. 	
\bibitem{1}J. Chen, Y. Wang, B. Xia, Characterization of subgroup perfect codes in Cayley graphs, Discrete Math. 343 (5) (2020) 111813.

\bibitem{cyclic} R. Feng, H. Huang, S. Zhou, Perfect codes in circulant graphs, Discrete Math. 340 (2017) 1522-1527.
\bibitem{2coloring} D. G. Fon-Der-Flaas, Perfect 2-colorings of a hypercube, Siberian Math. J. 48 (4) (2007) 740-745.
\bibitem{godsil}C. Godsil, G. Royle, Algebraic Graph Theory, Graduate Texts in Mathematics, Springer, New York (2001).
\bibitem{0}H. Huang, B. Xia, S. Zhou, Perfect codes in Cayley graphs, SIAM J. Discrete Math. 32 (2018) 548-559.
\bibitem{paper}Y. Khaefi, Z. Akhlaghi, B. Khosravi. On the subgroup perfect codes in Cayley graphs, Des. Codes Cryptogr. (2022) 1-7.

\bibitem{konig} D. K$\ddot{o}$nig, $\ddot{U}$ber Graphen und ihre Anwendung auf Determinantentheorie und Mengenlehre. Math. Ann., 77 (4) (1916) 453-465.
\bibitem{9}X. Ma, G.L. Walls, K. Wang, S. Zhou, Subgroup perfect codes in Cayley graphs, SIAM J. Discrete Math. 34 (3) (2020) 1909-1921.
\bibitem{hamming2coloring2}I. Mogilnykh and A. Valyuzhenich, Equitable 2-partitions of the Hamming graphs with the second eigenvalue, Discrete Math. 343 (11) (2020) 112039.

\bibitem{n} A. Neumaier, Completely regular codes, Discrete Math., 106/107 (1992) 353-360.
\bibitem{regularset} Y. Wang, B.Z. Xia, S.M. Zhou, Regular sets in Cayley graphs, J. Algebr. Comb. (2022).
\bibitem{subregularset}Y. Wang, B.Z. Xia, S.M. Zhou, Subgroup regular sets in Cayley graphs, Discrete Math. 345 (11) (2022) 113023.
\bibitem{on}J. Zhang, S. Zhou, On subgroup perfect codes in Cayley graphs, European J. Combin. 91 (2021) 103228. 
\bibitem{note}J. Zhang,  Y. Zhu, A note on regular sets in Cayley graphs, Bull. Aust. Math. Soc. (2023) 1-5.
\bibitem{total}S. Zhou, Total perfect codes in Cayley graphs, Des. Codes Cryptogr. 81 (2016) 489-504.
\end{thebibliography}
\end{document}